\documentclass[a4paper,11pt,reqno]{amsart}

\usepackage{latexsym,dsfont,amssymb,amsmath,amsthm,color}

\chardef\bslash=`\\ 

\numberwithin{equation}{section}

\newtheorem{theorem}{Theorem}[section]
\newtheorem{corollary}[theorem]{Corollary}
\newtheorem{lemma}[theorem]{Lemma}
\newtheorem{proposition}[theorem]{Proposition}

\theoremstyle{remark}
\newtheorem{remark}[theorem]{Remark}
\newtheorem{example}[theorem]{Example}
\theoremstyle{definition}

\newcommand\bp{\begin{proof}}
\newcommand\ep{\end{proof}}

\newcommand{\res}{\operatorname{res}}

\newcommand{\inv}{^{-1}}

\newcommand{\nxnx}{{\mathbb N \rtimes \mathbb N^\times}}

\newcommand{\N}{\mathbb N}
\newcommand{\Z}{\mathbb Z}
\newcommand{\Q}{\mathbb Q}

\newcommand{\OO}{\mathcal O}
\newcommand{\B}{\mathcal B}

\newcommand{\Prim}{\operatorname{Prim}}
\newcommand{\Prime}{\operatorname{Prime}}
\newcommand{\ind}{\operatorname{ind}}

\newcommand{\I}{\mathcal I}

\newcommand\af{{\mathbb{A}_f}}

\newcommand{\cA}{\mathfrak A}

\newcommand{\cP}{\mathcal P}
\newcommand{\cQ}{\mathcal Q}

\newcommand{\cT}{\mathfrak T}

\newcommand{\rxrx}{{R \rtimes R^\times}}
\newcommand{\rx}{R^{\times}}
\newcommand{\rhat}{{\hat R}}
\newcommand{\kxkx}{{K\rtimes K^*}}
\def\oma{\Omega_\af}
\def\omr{\Omega_\rhat}

\begin{document}

\title[Primitive ideal space of $\mathfrak T(R)$]{The primitive ideal space of the C*-algebra of \\the affine semigroup of algebraic integers} 
\author{Siegfried Echterhoff}
\date{26 January 2012}
\author[M. Laca]{Marcelo Laca}
\address{Siegfried Echterhoff, Mathematisches Institut, Einsteinstr. 62, 48149
M\"unster, Germany}
\email{echters@uni-muenster.de}
\address{Marcelo Laca, Department of Mathematics
and Statistics, P.O.B. 3060, University of Victoria,
Victoria, B.C. Canada V8W 3R4}  
\email{laca@uvic.ca}

\begin{abstract}
The purpose of this note is to give a complete description of the primitive ideal space of the  {C*-algebra $\cT[R]$ associated} to 
the ring of integers $R$ in a number field $K$  in the recent paper \cite{CDL}.
As explained in \cite{CDL}, $\cT[R]$ can be realized as {the Toeplitz C*-algebra of the affine semigroup $\rxrx$ over $R$ and} as a full corner of a crossed product
$C_0(\oma)\rtimes K\rtimes K^*$, where $\oma$ is a certain {adelic space.}
Therefore $\Prim(\cT[R])$ is homeomorphic to the primitive ideal space of  
this crossed product. Using a recent result of Sierakowski
together with the fact that every quasi-orbit for the action of $K\rtimes K^*$ on $\oma$ contains at least one point with trivial stabilizer we show that $\Prim(\cT[R])$ is homeomorphic to the quasi-orbit space for the 
action of $K\rtimes K^*$ on $\oma$, which in turn may be identified with  
 the power set $2^{\mathcal P}$ of the set of 
prime ideals $\mathcal P$  of R equipped with the power-cofinite topology.

\end{abstract}
  \thanks{The research for this paper was partially supported by the Deutsche Forschungsgemeinschaft 
(SFB 878) and by the Natural Sciences and Engineering Research Council of Canada}

\maketitle
\section{Introduction}
Let $R$ be the ring of integers in a number field $K$. 
In the recent paper \cite{CDL} Cuntz, Deninger and Laca  introduce an  algebra $\cT[R]$  attached to $R$ and they show that this algebra has a very interesting KMS-structure relative to a certain natural one parameter group of automorphisms.
The algebra $\cT[R]$ is an extension of the C*-algebra $A[R]$ studied previously in
\cite{C}, \cite{CL}, \cite{CL1} and, in contrast to $A[R]$, it is functorial under homomorphisms of rings.
While $A[R]$ is always simple, the new algebra $\cT[R]$ has a fairly rich ideal structure, and it is the aim of this note to give a detailed description of the primitive ideal space $\Prim(\cT[R])$ as a topological space. Since the closed ideals in a C*-algebra $A$ correspond bijectively and {inclusion-preserving} to the open subsets of $\Prim(A)$, we obtain a complete picture of the ideal structure of $\cT[R]$. 
A related extension $\mathcal T(\nxnx)$ of the C*-algebra $\OO_\N$ from \cite{C} had been introduced in \cite{LR-adv}, and although strictly speaking it is not equal to $\cT[\Z]$, its structure is similar enough
for our methods to apply  there as well.

It is shown in \cite{CDL} that
$\cT[R]$ is a full corner in a crossed product $C_0(\oma)\rtimes K\rtimes K^*$, where $\oma$ is a certain quotient space of the product $\mathbb A_f\times \mathbb A_f$ of the finite adele space $\mathbb A_f$ {over $K$} by itself. Since Morita equivalent C*-algebras have homeomorphic primitive ideal spaces, it therefore suffices to describe the primitive ideal space of this crossed product.
The situation is somewhat similar to the computation of the primitive ideal space of the original Bost-Connes algebra as performed by Laca and Raeburn in \cite{LR}. {Since the Bost-Connes  algebra 
is Morita equivalent to  the crossed product $C_0(\mathbb A_f)\rtimes \Q_+^*$ and since $\Q_+^*$ is an abelian group one can use  a theorem of Williams based on the Mackey machine  (e.g. see \cite[Theorem 1.1]{LR})  to compute  the primitive ideal space of the crossed product as a certain quotient of $\mathbb A_f\times \widehat{\Q_+^*}$, where $\widehat{\Q_+^*}$ denotes the Pontrjagin dual of $\Q_+^*$.}

The situation  becomes a bit more complicated in {the present case} since the $ax+b$-group $K\rtimes K^*$ is not abelian and the action of $K\rtimes K^*$ on $\oma$ has wildly varying stabilizer groups. However, we shall see in \S 3 below that every quasi-orbit for the action of $K\rtimes K^*$ on $\oma$ contains  at least one orbit with trivial stabilizers. It turns out that this suffices to prove that $\Prim(C_0(\oma)\rtimes K\rtimes K^*)$ is homeomorphic to the quasi-orbit space $\oma/\!\sim$ via an induction procedure (recall that  if $G$ is a group acting on a topological space $X$, then two elements $x,y\in X$ lie in the same {\em quasi-orbit} if $\overline{Gx}= \overline{Gy}$).  We deduce this fact in \S 2 in a much more general setting 
from a recent result of Sierakowski  \cite{Sier} which extends earlier results of Renault  \cite{Renault} and of Archbold and Spielberg \cite{AS}.

We  give a precise description of the quasi-orbit space $\oma/\!\sim$ in \S 3, showing that it is homeomorphic to
the power set $2^{\mathcal P}$ of the set of 
prime ideals $\mathcal P$  of R equipped with the power-cofinite topology.

\section{The primitive ideal space of crossed products by essentially free actions}\label{GRST}

Let $G$ be a {countable discrete group acting by automorphisms on a separable} C*-algebra $A$. We {will} give a detailed description of the primitive ideal space of the reduced crossed product 
$A\rtimes_rG$ in the case where the action {of $G$ on $A$} is exact and the {associated} action of $G$ on the space 
$\widehat{A}$ of equivalence classes of irreducible representations of $A$ is  {\em essentially free} in the sense that every closed subset $C$  of $\widehat{A}$  contains a dense invariant  subset $D$ such that $G$ acts freely on $D$. Our method builds up on recent work of Sierakowski \cite{Sier}, {which relies} on a central result of the  paper \cite{AS} 
by Archbold and Spielberg and {extends} earlier results by Renault (see \cite[Corollary 4.9]{Renault}), 

Recall that an action of the group $G$ on $A$ is {\em exact}, if for any $G$-invariant ideal $I\subseteq A$ the 
corresponding sequence of reduced crossed products
$$0\to I\rtimes_rG\to A\rtimes_r G\to (A/I)\rtimes_rG\to0$$
is exact. The group $G$ is called {\em exact} if every action of $G$ on any C*-algebra $A$ is exact. 
The class of exact groups is quite large. It contains  all amenable groups and all countable groups which can be embedded (as abstract groups) in {a} linear group over some field $K$ (in particular all free groups, see \cite{GHW}). We refer to the discussion in \cite{Sier} for further details.

For a C*-algebra $A$ we denote by  $\mathcal I(A)$ the set of closed ideals in $A$ equipped with the
 {\em Fell topology} for which a sub-base of open sets is given by the sets of the form
$$U_I:=\{J\in \mathcal I(A): I\not\subseteq J\},\quad I\in \mathcal I(A).$$
Restricted to $\Prim(A)$, this is the usual Jacobson topology. 
Note that if  $A$ and $B$ are two C*-algebras 
and $\varphi:\mathcal I(A)\to \mathcal I(B)$ is any map which preserves inclusion of ideals, 
then $\varphi$ is continuous.
If $(A,G,\alpha)$ is a system, we denote by $\I^G(A)\subseteq \I(A)$ the set of $G$-invariant closed ideals in $A$ equipped with the topology restricted from $\I(A)$. 
We will need the following formulation of Sierakowski's results. 

\begin{theorem}[{Sierakowski}]\label{Sierakowski} 
Suppose that $(A,G,\alpha)$ is a dynamical system with $G$ discrete such that the action of $G$ on $A$ is exact and the  action of $G$ on $\widehat{A}$ is essentially free. Then 
the map
$$\res: \I(A\rtimes_rG)\to \I^G(A); \quad J\mapsto J\cap A$$ is 
a homeomorphism with inverse map given by $$\ind:\I^G(A)\to \I(A\rtimes_rG); \quad \ind(I)=I\rtimes_rG.$$
\end{theorem}

\begin{proof} This follows from  combining \cite[Proposition 1.3]{Sier} with \cite[{Theorem 1.20}]{Sier}
and observing that 
since both maps obviously preserve {inclusion} of ideals, they are homeomorphisms with respect to the Fell-topologies.
\end{proof}

\begin{remark} (1) The assumption that the action of $G$ on $\widehat{A}$ is essentially free is certainly weaker than the condition that the action of $G$ on $\Prim(A)$ is essentially free. This follows directly from the definitions {of essential freeness} and of the Jacobson topologies on $\widehat{A}$ and $\Prim(A)$ (e.g. see \cite[Chapter 3]{Dix}).

 (2) It is shown in  \cite[\S 2]{Sier} that the assumption that the action of $G$ on $\widehat{A}$ is essentially free can be replaced by a somewhat weaker assumption which Sierakowski calls the {\em residual Rokhlin* property}. Thus all results discussed below will remain true if the assumption that the action of $G$  on $\widehat{A}$ is essentially free is replaced by the assumption that the action satisfies the residual Rokhlin* property.
 
\end{remark}

Suppose that $(A,G,\alpha)$ is a C*-dynamical system. A $G$-invariant ideal 
$I\subseteq A$ is called {\em $G$-prime}  if 
for any pair of closed $G$-invariant ideal $J_1, J_2\subseteq A$ with $J_1\cap J_2\subseteq I$ we have 
$J_i\subseteq I$ for $i=1$ or $i=2$. We denote by $\Prime^G(A)$ the set of $G$-prime ideals in $A$ and we denote by $\Prime(A)$ the set of prime ideals in $A$ (the case where $G$ is the trivial group).
The Fell-topology on $\I(A)$ induces topologies on $\Prime^G(A)$ and $\Prime(A)$.

\begin{proposition}\label{prop-prime}
Suppose that $(A,G,\alpha)$ satisfies the assumptions of Theorem \ref{Sierakowski}. Then the map
$\res:\I(A\rtimes_r G)\to \I^G(A)$ restricts to a homeomorphism $\Prime(A\rtimes_rG)\cong \Prime^G(A)$. 
\end{proposition}
\begin{proof} Suppose that $I\in\Prime(A\rtimes_rG)$ and let $J_1, J_2\in \I^G(A)$ such that 
$J_1\cap J_2\subseteq I\cap A$. It follows then from Theorem \ref{Sierakowski} that 
$J_1\rtimes_rG\cap J_2\rtimes_rG\subseteq (I\cap A)\rtimes_rG=I$ and hence that 
$J_i\rtimes_rG\subseteq I$ for some $i\in \{1,2\}$. But then we also have 
$J_i=(J_i\rtimes_rG)\cap A\subseteq I\cap A$, which proves that $I\cap A$ is $G$-prime.
A very similar argument shows that $J\rtimes_rG$ is prime if $J\in \Prime^G(A)$.
The result then follows from Theorem \ref{Sierakowski}.
\end{proof}

Recall that the primitive ideals in a C*-algebra are the kernels of the irreducible representations of $A$. If
$(A,G,\alpha)$ is a system, then an ideal $I\in \I^G(A)$ is called {\em $G$-primitive} if $I=P^G:=\cap_{g\in G}\alpha_g(P)$ for some primitive ideal $P$ of $A$. We denote by $\Prim^G(A)$ the set of $G$-primitive ideals equipped with the Fell-topology. 

Moreover, if a group $G$ acts on a topological space $X$, then the {\em quasi-orbit space} $\OO(X)$ is defined as the quotient space $X/\sim$ by the equivalence relation $x\sim y\Leftrightarrow \overline{G x}=\overline{G y}$. In what follows we denote by $\OO(x)$ the {\em quasi-orbit of $x$}, i.e., the equivalence class of $x$ under this relation.
The following lemma is \cite[Lemma on p. 221]{Green}:

\begin{lemma}[{Green}]\label{primG}
Suppose that $(A,G,\alpha)$ is a system. Then 
$\OO(\Prim(A))\to \Prim^G(A);$  ${\OO(P)}\mapsto P^G$
is a homeomorphism
\end{lemma}

The following lemma is also well known to the experts.  For completeness we give a proof, which is an adaptation of the proof of \cite[Corollary on p. 222]{Green}.

\begin{lemma}\label{lem-prime}
Let $(A,G,\alpha)$ be a system with $A$ separable.  Then $\Prime^G(A)=\Prim^G(A)$.
\end{lemma}
\begin{proof} We first remark that every primitive ideal is prime, since
if $P=\ker\pi$ for some irreducible representation $\pi:A\to \mathcal B(H)$ and if $J_1,J_2\in \I(A)$ such that $J_1\cap J_2\subseteq P$,  then $J_i\subseteq P$ for some $i\in \{1,2\}$ because otherwise  $\pi$ would restrict to irreducible representations on $J_1$ and $J_2$ and then 
$\{0\}=\overline{\pi(J_1\cap J_2)H}=\overline{\pi(J_1)\pi(J_2)H}=H.$

Suppose now that $I=P^G$ is $G$-primitive. Then, if $J_1,J_2\in \I^G(A)$ 
such that $J_1\cap J_2\subseteq I$, we have $J_1\cap J_2\subseteq P$, and hence $J_i\subseteq P$ for some $i\in \{1,2\}$. But then $J_i=J_i^G\subseteq P^G=I$. Hence $I$ is also $G$-prime.

For the converse 
let $I\in \Prime^G(A)$. We need to show that there exists a $P\in \Prim(A)$ with $I=P^G$. For this let $F\subseteq \Prim(A)$ be the hull of $I$, i.e., $F=\{Q\in \Prim(A): I\subseteq Q\}$. This is a closed $G$-invariant subset of $\Prim(A)$ and we need to show that
$F=\overline{G(P)}$ for some $P\in \Prim(A)$. This means that the image $C:=q(F)\subseteq \mathcal O(\Prim(A))$ is the closure of a single point $\mathcal O(P)\in \mathcal O(\Prim(A))$, where $q:\Prim(A)\to\mathcal O(\Prim(A))$ denotes the quotient map.
 It follows from the discussion preceding \cite[Lemma on p. 222]{Green}  that $\mathcal O(\Prim(A))\cong \Prim^G(A)$ is a totally Baire space, and it is second countable since $A$ is separable.
By \cite[Lemma on p. 222]{Green}, {to conclude that $F=\overline{G(P)}$} it suffices to show that $C$ is irreducible, which means that it cannot be written as a union of two {proper closed subsets $C_1, C_2$}.
So assume that $C_1,C_2$ are closed subsets of $\OO(\Prim(A))$ with $C=C_1\cup C_2$. Let $F_1, F_2$ denote their inverse images in $\Prim(A)$ and let $J_i=\ker F_i=\cap\{Q: Q\in F_i\}$ for $i=1,2$. Then $F=F_1\cup F_2$ which implies that $J_1\cap J_2=I$. Since $I$ is $G$-prime we have $J_i\subseteq I$ for some $i\in \{1,2\}$ which implies that 
$F_i=\{Q\in \Prim(A): J_i\subseteq Q\}$ contains $F$. This finishes the proof.
\end{proof}

\begin{corollary}\label{cor-homeo}
Suppose that $(A,G, \alpha)$ is a system such that $A$ is separable, $G$ is countable, {the action of $G$ on $A$ is exact} and the action of $G$ on $\widehat{A}$ is essentially free. Then 
$$\ind: \Prim^G(A)\to \Prim(A\rtimes_rG);\quad   I\mapsto I\rtimes_rG$$
is a homeomorphism with inverse map given by 
$$\res:\Prim(A\rtimes_rG)\to \Prim^G(A); \quad P\mapsto P\cap A.$$
\end{corollary}
\begin{proof} This follows as a direct combination of Proposition \ref{prop-prime} together with Lemma \ref{lem-prime} (which also implies that $\Prim(A\rtimes_rG)=\Prime(A\rtimes_rG)$ since $A\rtimes_rG$ is separable).
\end{proof}

Recall that if $(A,G,\alpha)$ is a system with $G$ discrete, then every  $*$-representation
$\pi:A\to \B(H)$ gives rise to an induced representation $\ind \pi: A\rtimes_rG\to \B(l^2(G,H))$
which is the integrated form $\tilde\pi\times \lambda$ of the covariant homomorphism $(\tilde\pi,\lambda)$ of $(A,G,\alpha)$ on $l^2(G,H)$ given by the formulas
$$\big(\tilde\pi(a)\xi\big)(g)=\pi(\alpha_{g^{-1}}(a))\xi(g)\quad\text{and}\quad 
\big(\lambda(h)\xi\big)(g)=\xi(h^{-1}g),  \quad \forall \xi\in l^2(G,H), a\in A, h,g\in G.$$

 If $J=\ker\pi$ is $G$-invariant, then $G$ acts on $A/J$ and $\ker \tilde\pi = J$, hence  $\ind\pi$ gives a faithful representation of $A/J\rtimes_rG$. When the action of $G$ on $A$ is exact, we see that $\ker(\ind\pi)=J\rtimes_rG$. 
 It can be shown  (e.g., see \cite[Proposition 9]{Green}) that in general $\ker(\ind\pi)$ depends only on $\ker \pi$ hence
 the definition of induced ideals given earlier for  $G$-invariant ideals can be extended 
 to all ideals of $A$ by letting
 $\ind J = \ker (\ind\pi)$, where $\pi$ is any representation of $A$ with $\ker \pi = J$.

Moreover, it is easily checked that for a general representation $\pi:A\to\B(H)$ the induced representation 
 $\ind\pi$ is unitarily equivalent to $\ind(\pi\circ \alpha_g)$ for all $g\in G$ (the equivalence being  implemented by the unitary $\lambda(g)$). Thus we see that on the level of ideals we get for all $g\in G$ and for every ideal $J\in \I(A)$ that
 $$\ind J=\ind (\alpha_g(J))=\ind(J^G).$$
 Combining this with Corollary \ref{cor-homeo} and Lemma \ref{primG} 
 we get {the following corollary.}
 
 \begin{corollary}\label{cor-quasi-orbit}
 {Suppose that $(A,G, \alpha)$ is a system such that $A$ is separable, $G$ is countable, {the action of $G$ on $A$ is exact} and the action of $G$ on $\widehat{A}$ is essentially free.} Then
$$\ind:\OO(\Prim(A))\to\Prim(A\rtimes_rG); \quad \OO(P)\mapsto \ind P$$
is a homeomorphism.
 \end{corollary}

\section{Orbits, stabilizers, and quasi-orbits of an adelic space.}

Let $K$ be an algebraic number field with ring of algebraic integers $R$.  
The nonzero elements of $R$ form a semigroup denoted $\rx$, which contains 
the group $R^*$ of invertible elements or {\em units} of $R$.  
In this section we will study the primitive ideal space of the C*-algebra
$\cT[R]$ which was introduced in \cite{CDL} in terms of generators and relations
and shown to be isomorphic to the C*-algebra of the left regular representation of the semigroup 
$\rxrx$ on $\ell^2(\rxrx)$.
Because of  Propositions 5.1 and 5.2  of \cite{CDL}, there is an action of the group $\kxkx$ on a locally compact Hausdorff space $\oma$ such that 
$\cT[R]$ is isomorphic to a full corner in 
the transformation group C*-algebra $C_0(\oma) \rtimes \kxkx$,
and thus 
its primitive ideal space is naturally homeomorphic to 
that of $C_0(\oma) \rtimes \kxkx$. 
{For a  specific description of the homeomorphism  see \cite[Lemma 2.7]{LR}.}
   We want to use the results from 
    Section~\ref{GRST} above, for which it becomes important to
understand the orbits, the stability subgroups, and the quasi-orbit space
of the action of $\kxkx$ on $\oma$. 

We begin by recalling the definition of the space $\oma$ and the action of $\kxkx$ 
from section~5 of \cite{CDL}.
Denote by $\af$ the locally compact, totally disconnected ring of finite adeles over $K$.
The finite adeles have a maximal compact subring  $\hat R$, the {\em integral adeles}, whose group of units (i.e. invertible elements)  is 
denoted $\hat R^*$, the compact group of {\em integral ideles}.
The space $\oma$ is the quotient of $\af \times \af $ by the equivalence
\[
(r,a) \sim (s, b) \quad \iff \quad a\hat R^* = b \hat R^* \text{ and }  r-s \in a\hat R.
\]
If we denote the class of $(r,a) \in \af\times\af$  by $\omega_{r,a}$, then the group $\kxkx$ acts on $\oma$  in the obvious way by 
$(x,k) \cdot \omega_{r,a} = \omega_{x+kr, ka}$.

Denote by $\cP$ the set of all prime ideals of $K$.
Recall from section 5 of  \cite{CDL} that each point $ \omega_{r,a} \in \oma$ has a valuation vector
$\{v_P(a)\}_{P\in \cP}$  {which depends only on (and in fact characterizes) the second component}. {This valuation vector can be defined as follows:} if  $a\in \hat R$, let $v_P(a)$ be the smallest  
 nonnegative integer $n$ such that the canonical projection of $a$ in 
 $R/P^{n+1}$ is nonzero, and put  $v_P(a) = \infty$ if $a$ projects onto $0 \in R/P^n$ for every $n$. If $a$ is a finite adele, then there exists $d \in R$ such that $da \in \hat R$, and we let 
 $v_P(a) = v_P(da) - v_P(d)$, which does not depend on $d$.
 
{Thus, we may regard  the second component of a point of $\oma$ as a {\em superideal} for which, in analogy to the supernatural numbers, we allow infinite powers and also finitely many negative ones.} Hence we write
\[
\af /\hat R^* \cong \prod _{P \in \cP}(P^{\Z \cup \{\infty\}}; P^{ \N \cup \{\infty\}}), \qquad a \mapsto \prod_\cP P^{e_P(a)}
\]
where the restriction in the direct product is that the elements of $\af/\hat R^*$ can have negative exponents 
for at most finitely many prime ideals. The topology in each coordinate is from the natural 
order in $\Z$; as usual, the factor $P^\infty $ corresponds to $0 \in K_P$ in the $P$ coordinate.

  \begin{lemma}\label{orbitclosure}
 { For each  $a \in \af/\hat R^*$, let $Z(a): = \{ P\in \cP:  a_P =0\} = \{P\in \cP: e_P(a) = \infty\}$.
  Then the closure of the orbit of $\omega_{r,a}$ is the set 
  $\{\omega_{s,b} \in \oma: Z(b) \supset Z(a)\}$.}
  \end{lemma}
  \begin{proof}
  It is easy to see that the set $\{\omega_{s,b} \in \oma: Z(b) \supset Z(a)\}$
  is closed and that every point in the orbit of $\omega_{r,a}$ 
 has a second coordinate that vanishes on $Z(a)$
  (and possibly also at other primes). Thus $\{\omega_{s,b} \in \oma: Z(b) \supset Z(a)\}$ 
 contains the closure of the orbit of $\omega_{r,a}$.

To prove the reverse inclusion, {consider} $\omega_{r,a} \in \oma$ and let $(s,b)$ be a point in $\af\times (\af/\hat R^*)$ with 
$Z(b) \supset Z(a)$.
{Let $W_2$ be typical basic neighborhood of $b \in \af/\hat R^*$:
\[
W_2 := \{ c\in \af/ \hat R^*: v_{P_j}(c) = v_{P_j}(b)  \text{ for } P_j \notin Z(b) \text{ and } 
v_{P_j}(c)  \geq n_j  \text { for } P_j \in Z(b)\} 
\]
where
$\{P_1, P_2, \cdots ,P_n\}$ is a given finite set of prime ideals  and $n_j \geq 0$  is a given integer for every $j$ such that $P_j \in Z(b)$. }
For each $j$ such that $P_j$ is not in $Z(b)$ choose $e_j=v_{P_j}(b)-v_{P_j}(a)$; 
  for each $j$ such that $P_j \in Z(b)\setminus Z(a)$, let 
  $e_j = n_j + | v{P_j}(a)|$; finally let $e_j = 0$ for each $j$ such that $P_j \in Z(a)$ (this last choice is not essential,  we just make it to be thorough).

Now let $Q$ be a {prime} ideal in the inverse class of 
 the product  $ \prod_{j=1}^n P_j^{e_j} $ such that $Q \neq P_j$ for $j = 1, 2, \cdots, n$ and choose $k$ to be a generator of the 
 principal ideal $Q \prod_{j=1}^n P_j^{e_j} $. Then $k a \in W_2$.
 
 {If $W_1$ is any neighborhood of $s$ in $\af $, there exists $x\in K$ such that 
 $x + kr \in W_1$, because $K$ is dense in $\af$. Since (the images of) neighborhoods of the form
 $V = W_1 \times W_2$ form a neighborhood basis for 
 $\omega_{s,b}$  in the quotient $\oma$, and since we have shown that every such neighborhood contains a point of the form $\omega_{x+kr, ka}$, we conclude that 
$\omega_{s,b}$ is in the closure of the orbit $\{\omega_{x+kr, ka}: x\in K, k \in K^*\}$
  of $\omega_{r,a}$.}  \end{proof}

It is now easy to describe the quasi-orbit space. 
Following Section 2 in \cite{LR} we consider
 the power set $2^\cP$ of the set of prime ideals of $R$ equipped with the power-cofinite topology, in which the basic open sets are
indexed by  finite $G\subseteq \cP$ and are given by $U_G = \{ T \in 2^\cP:  T\cap G =\emptyset\}$. 
\begin{lemma}\label{lem-quasiorbit}
 The correspondence $ \omega_{r,a} \mapsto Z(a)$
induces a homeomorphism of the quasi-orbit space to the power set of the prime ideals
 with the power-cofinite topology.
 \end{lemma}

 \begin{proof}
 The preceding lemma shows that 
 two points are in the same quasi-orbit if and only if their second components vanish at exactly the same subset of prime ideals. This gives a bijection from the quasi-orbit space to $2^\cP$.
  The topology on the quasi-orbit space is obtained from the topology on $\oma$ through the quasi-orbit map $\omega_{r,a} \mapsto q(\omega_{r,a})$
 which is continuous and open by \cite[p.221]{Green}.
 
 Since the first component is lost in passing from $\omega_{r,a}$ to $Z(a)$,
  an argument similar to that in the proof of 
 \cite[Proposition 2.4]{LR} shows that 
 basic neighborhoods in $\oma$ are mapped exactly onto
 the power-cofinite neighborhoods of $2^\cP$. 
   \end{proof}
 
  \begin{example} The action of $\kxkx$ has many points with nontrivial stabilizer. Here we give some examples of such points together with their stabilizers.

(i) Let $(x,k) \in \kxkx$ and let $a\in \rx := R\setminus \{0\}$. Then the point  $(x,k) \omega_{0,1} = \omega_{x,k}$ coincides with $\omega_{0,1}  \in \oma$ 
 iff  $k \in \hat R^* \cap K = R^*$ and $x \in \hat R \cap K = R $, so the stabilizer of $\omega_{0,1} \in \oma$ is $R\rtimes R^*$.
 
(ii) The points $\omega_{m,k}$ in the orbit of $\omega_{0,1}$  have stabilizers 
$\{ (m,k) (r,w) (-mk\inv, k\inv): r\in R, \ w\in R^*\} = \{ (m (1-w) +kr, w):  r\in R, \ w\in R^*\}$. 

(iii) The stabilizer of the  point $\omega_{0,0}$ consists of elements $(x,k)$ such that 
$x = 0 \in (\af \cap K) /0= K$. That is, $\{0\} \times K^*$. 

(iv) The  points $\omega_{x,0}$ in the orbit of $\omega_{0,0}$ have stabilizers
$(x,1) (\{0\} \times K^*) (-x,1) = \{(x(1-k), k): k \in K^*\}$.
\end{example}

Despite the abundance of nontrivial stabilizers, we shall see next that 
there are sufficiently many points with trivial stabilizer 
to generate all quasi-orbits.
\begin{lemma}\label{trivialstabil}
For every subset $A$ of prime ideals there exists 
a point $\omega_{r,a}$ with trivial stabilizer and $Z(a) = A$.
\end{lemma}
\begin{proof}
If $P\in A$ the condition $\omega_{x+kr, ka} = \omega_{r,a}$ implies
$x+(k-1) r_P = 0$, that is, either $r_P = -x/(k-1) \in K$, or else  $k=1$ and $x=0$. 
When $A \neq \emptyset$ the first case can be eliminated by choosing
 $\omega_{r,a}$ with $r_P \notin K$. 

For $A =\emptyset$ choose any $a$ with $Z(a) = \emptyset$
such that $\cQ := \{ P\in \cP : v_P(a) >0\}$ is infinite.
Write $\cQ = \cQ_1 \sqcup \cQ_2$ with both  $\cQ_1$ and $ \cQ_2$
infinite and choose any $r$ such that 
 $v_P(r) = v_P(a) - 1$ for $P\in \cQ_1$ and 
 $v_P(r) = v_P(a) $ for $P\in \cQ_2$.
 
The group element $(x,k)$ fixes the point $\omega_{r,a}$ if and only if
$ka\hat R^* = a \hat R^*$ (so that $k\in R^*$)
and $x+(k-1) r \in a\hat R$. If $k\neq 1$, this last condition gives
$r \in \frac{1}{k-1} (a\hat R - x)$.
When $x=0$ this  is clearly impossible 
because the factor $\frac{1}{k-1}$ only reduces the valuation at finitely many places.
When $x\neq 0$ this is also impossible because subtracting $x$ reduces to zero the valuation 
 at all primes in $\cQ$ except possibly  at the finitely many  primes $P$
that satisfy $v_P(k) \neq 0$.

\end{proof}

As a direct corollary of the above result we get {the following.}
\begin{corollary}\label{cor-orbit}
The action of $\kxkx$ on $\oma$ is essentially free.
\end{corollary}
\begin{proof} 
The points $\omega_{r,a}$ considered in the last part of the proof of Lemma \ref{trivialstabil} (for the case $A=\emptyset$) have trivial stabilizers and 
their orbits are dense by Lemma \ref{orbitclosure}.
\end{proof}

We are now ready for our main result, namely the 
characterization of the primitive  ideal space of the C*-algebra $\cT[R]$ 
from \cite{CDL}. Recall that $\cT[R]$ can be realized as the
 full corner in the crossed product
$C_0(\oma)\rtimes K\rtimes K^*$ corresponding to the full projection 
$\mathbf 1_{\omr}$ associated to the clopen subset $\omr \subset \oma$ as in
Proposition 5.1 of \cite{CDL}.

\begin{theorem}\label{main}
Let $2^\cP$ denote the power set of the set of prime ideals of $R$ with the 
power-cofinite topology. 
For each subset  $A$ of $\cP$ let $I_A$ be the kernel of the
 {compression to $\cT[R]$ of the} induced representation
corresponding to a point $\omega_{r,a}$ with $Z(a) = A$ and trivial stabilizer.
Then the map $A \mapsto I_A$ is a homeomorphism of $ 2^\cP$ to
$\Prim \cT[R] $.
\end{theorem}
\begin{proof} {The compression by a full projection respects unitary equivalence and gives a homeomorphism of primitive ideals, see, e.g. Lemma 2.7 of \cite{LR}. Hence
the result is a direct consequence of Lemmas \ref{lem-quasiorbit} and \ref{trivialstabil}
and of Corollary \ref{cor-orbit} above.}
\end{proof}

\begin{remark}
Since the action of $K\rtimes K^*$ on $\oma$ is essentially free, Sierakowski's theorem (Theorem \ref{Sierakowski}) also gives a direct correspondence between the ideals of $\cT[R]$ and the invariant ideals in $C_0(\oma)$, i.e., the invariant open sets in $\oma$, which are in  bijection to the open subsets of $\mathcal O(\oma)\cong 2^{\cP}$, since every invariant open set is the union of the quasi-orbits it contains.
\end{remark}

\begin{remark} {In the situation of Theorem 2.8 of \cite{LR}, where only the multiplicative action of $\Q_+^*$ on the  finite (rational) adeles is considered, the nontrivial isotropy appears in the primitive ideal space through the infinite torus $\widehat{\Q_+^*}$ sitting above the only point with nontrivial isotropy.  This had originally led us to believe that $\cT[R]$ would have a complicated primitive ideal structure. Thus  Lemma~\ref{trivialstabil} was a welcome surprise:
the stability subgroups do not play a role in the primitive ideal space of $\cT[R]$,
{ultimately due to the additive action of $R$.}}
\end{remark}

\begin{remark} The description of the primitive ideal space shows that there is only one maximal ideal in $\cT[R]$, namely the primitive ideal $I_\cP$ corresponding to the set $\cP$ viewed as an element of $2^\cP$. 
Therefore, the quotient $\cT[R]/I_{\cP}$ is the only simple quotient of $\cT[R]$, which is isomorphic to the ring C*-algebre $\cA[R]$ considered by Cuntz and Li in \cite{CL1}.
\end{remark}


\begin{thebibliography}{10}

\bibitem{AS}

R.J. Archbold and J.S. Spielberg.
\emph{Topological free actions and ideals in discrete C*-dynamical systems.} Proc. Edinburgh Math. Soc. {\bf 37} (1993), 119--124.

\bibitem{C} {J. Cuntz. \emph{C*-algebras associated with the $ax+b$-semigroup over N.} G. Corti\~nas, 
(ed.) et al., K-theory and noncommutative geometry. Proceedings of the ICM 2006 satellite
conference, Valladolid, Spain, August 31–September 6, 2006. Z\"urich: European Mathematical
Society (EMS). Series of Congress Reports, 201-215 (2008)., 2008.}

\bibitem{CL} {J. Cuntz and X. Li. {\em The regular C*-algebra of an integral domain. } Clay Mathematics
Proceedings, {\bf 10} (2010), 149–170.}

\bibitem{CL1}  {J. Cuntz and X. Li. {\em C*-algebras associated with integral domains and crossed products
by actions on adele spaces.}  J. Noncommut. Geom., {\bf 5}  (2011), 1–37.}

\bibitem{CDL}
J. Cuntz, C. Deninger and M. Laca.
\newblock \emph{$C^*$-algebras of Toeplitz type associated with algebraic number fields.}
\newblock Preprint, 2011.

\bibitem{Dix}
J. Dixmier. C*-algebras. 
North-Holland Mathematical Library, Vol. 15. North-Holland Publishing Co., Amsterdam-New York-Oxford, 1977.


%

\bibitem{Fell1}
J.M.G. Fell.
\newblock {\em The dual spaces of $C^*$-algebras. }
\newblock Trans. Amer. Math. Soc.  {\bf 94}  (1960), 365--403. 

\bibitem{GHW} E. Guentner, N. Higson and S. Weinberger. 
\newblock\emph{The Novikov conjecture for linear groups.} 
\newblock Publ. Math. Inst. Hautes \'Etudes Sci. {\bf 101}  (2005), 243–268.

\bibitem{GR}
E. C. Gootman, J. Rosenberg.
\newblock {\em The structure of crossed product $C^* $-algebras: a proof of the generalized Effros-Hahn conjecture.}
\newblock Invent. Math. {\bf 52} (1979), no. 3, 283--298.

\bibitem{Green}
P. Green.
\newblock \emph{The local structure of twisted covariance algebras.}
\newblock Acta Math. {\bf 140} (1978), no. 3-4, 191--250.


\bibitem{LR}
M. Laca and I. Raeburn.
\newblock \emph{The  ideal structure of the Hecke $C^*$-algebra of Bost and Connes.}
\newblock Math. Ann. {\bf 318} (2000), 433-451.

\bibitem{LR-adv} M. Laca and I. Raeburn. \newblock \emph{Phase transition on the Toeplitz algebra of the affine semigroup over the natural numbers.} 
\newblock Adv. Math. {\bf 225} (2010), 643-688.

\bibitem{Renault}
J. Renault.  
\newblock\emph{The ideal structure of groupoid crossed product C*-algebras.} With an appendix by Georges Skandalis. 
\newblock J. Operator Theory {\bf 25} (1991), no. 1, 3--36.

\bibitem{Sier}
A. Sierakowski. 
\newblock \emph{The ideal structure of reduced crossed products.} 
\newblock M\"unster J. of Math. {\bf 3} (2010), 237--262.

\bibitem{Sauv}
J.-L. Sauvageot.
\newblock \emph{Id\'eaux primitifs induits dans les produits crois\'es.}
\newblock J. Funct. Anal.  {\bf 32}  (1979),  {no. 3}, 381--392.

             
\bibitem{Wi0}
D.P. Williams. 
\newblock \emph{The topology on the primitive ideal space of transformation group $C\sp{\ast} $-algebras and C.C.R. transformation group $C^*$-algebras.} 
\newblock Trans. Amer. Math. Soc. {\bf 266} (1981), no. 2, 335--359.


\bibitem{Dana-book} D.P. Williams. Crossed products of C*-algebras. Mathematical Surveys and Monographs, Vol 134, AMS 2007.


\end{thebibliography}
\end{document}